\documentclass{amsart}

\title{Coherence without unique normal forms}
\author{Jonathan A. Cohen}
\address{Department of Computing\\
  Macquarie University\\
  Sydney NSW 2109\\
  Australia
}
\email{jonathan.cohen@anu.edu.au}

\usepackage{amssymb, amsmath, amsfonts,mathrsfs,bussproofs,verbatim,rotating}
\usepackage[matrix,arrow,curve]{xy}
\CompileMatrices

\DeclareMathAlphabet{\mathpzc}{OT1}{pzc}{m}{it}

\newcommand{\C}{\ensuremath{\mathscr{C}}}
\newcommand{\F}{\ensuremath{\mathcal{F}}}
\newcommand{\T}{\ensuremath{\mathcal{T}}}
\newcommand{\tuple}{\ensuremath{\F,\theta_{\F}, \T,
    \theta_{\T}}}  

\newcommand{\Fr}{\ensuremath{\mathbb{F}}}
\newcommand{\Term}{\ensuremath{\mathrm{Term}}}

\newcommand{\N}{\ensuremath{\mathbb{N}}}
\newcommand{\Ob}{\ensuremath{\mathrm{Ob}}}
\newcommand{\Mor}{\ensuremath{\mathrm{Mor}}}
\newcommand{\of}{\ensuremath{\cdot}}
\newcommand{\M}{\ensuremath{\mathscr{M}}}

\newcommand{\tensor}{\ensuremath{\otimes}}
\renewcommand{\S}{\ensuremath{\mathpzc{S}}}

\newcommand{\Shape}{\ensuremath{\mathrm{Shape}}}
\newcommand{\var}{\ensuremath{\mathrm{Var}}}
\newcommand{\Red}{\ensuremath{\mathrm{Red}}}
\newcommand{\Sub}{\ensuremath{\mathrm{Sub}}}
\newcommand{\R}{\ensuremath{\mathbb{R}}}
\newcommand{\Out}{\ensuremath{\mathrm{Out}}}

\newtheorem{theorem}{Theorem}[section]
\newtheorem{lemma}[theorem]{Lemma}
\newtheorem{definition}[theorem]{Definition}
\newtheorem{proposition}[theorem]{Proposition}
\newtheorem{corollary}[theorem]{Corollary}

\newtheorem{example}[theorem]{Example}

\newcommand{\arrow}[2]{\ensuremath{
    \def\objectstyle{\scriptstyle}
    \def\labelstyle{\scriptstyle}
    \xymatrix@1{
      {#1} ~\to~  {#2}
    }
  }
}
\begin{document}

\begin{abstract}
  Coherence theorems for covariant structures carried by a category
  have traditionally relied on the underlying term rewriting system of
  the structure being terminating and confluent. While this holds in a
  variety of cases, it is not a feature that is inherent to the
  coherence problem itself. This is demonstrated by the theory of
  iterated monoidal categories, which model iterated loop spaces and
  have a coherence theorem but fail to be confluent. We develop a
  framework for expressing coherence problems in terms of term
  rewriting systems equipped with a two dimensional congruence. Within
  this framework we provide general solutions to two related coherence
  theorems: Determining 
  whether there is a decision procedure for the commutativity of
  diagrams in the resulting structure and determining sufficient
  conditions ensuring that ``all diagrams commute''. The resulting
  coherence theorems rely on neither the termination nor the
  confluence of the underlying rewriting system. We apply the theory
  to iterated monoidal categories and obtain a new, conceptual proof
  of their coherence theorem. 
\end{abstract}
\maketitle

\section{Introduction}

Coherence theorems are a mechanism for ensuring that an extra
structure carried by a category is not too wildly behaved. This
typically takes the form of an assurance that a certain large class of
diagrams always commutes. In the most favourable situation, one proves
that any diagram built solely out of the structuring functors and
natural transformations is guaranteed to commute. This was the case in
the earliest coherence results of Mac Lane for monoidal and symmetric
monoidal categories \cite{MacLane_natural}. 

A close examination of Mac Lane's proof reveals a connection between
covariant structures carried by categories and term rewriting
theory. In particular, the proof mainly revolves around elucidating
the fact that a free monoidal structure on a discrete category,
considered as a term rewriting system, is terminating and
confluent. ``Termination'' means that there are no infinite chains of
non-identity morphisms, while ``confluence'' is the property that
every span may be completed into a square (see Figure
\ref{fig:confluence}).

\begin{figure}[ht]
    \[
    \def\objectstyle{\scriptstyle}
    \def\labelstyle{\scriptstyle}
    \vcenter{
      \xymatrix{
        {\cdot} \ar[r] \ar[d] &  {\cdot} \ar@{-->}[d] \\
        {\cdot} \ar@{-->}[r] & {\cdot}
      }
    }
    \]
    \caption{Confluence}\label{fig:confluence}
\end{figure}
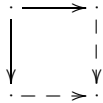 

Confluence and termination together conspire to ensure that a term
rewriting system has unique normal forms. That is, not only is every
chain of morphisms finite, but every sequence of morphisms beginning
from an object ends at a point that depends only on the starting
object. This seemingly strong property is present in a very large array
of structures and has, for instance, been exploited by Laplaza to
derive coherence theorems for directed associativity
\cite{laplaza:associative} and for distributive categories
\cite{laplaza:distributive}. 

Unfortunately, it is simply not the case that every coherent covariant
structure has unique normal forms. For instance, the structure
consisting of a unary functor $F$ and the single natural
transformation $F(X) \to F(F(X))$ is non-terminating, but easily seen
to be coherent. A more spectacular counterexample to the hope that
coherent structures have unique normal forms is provided by the theory
of iterated monoidal categories \cite{iterated}, which arise as a 
categorical model of iterated loop spaces and fail to be confluent. 

We are now faced with the problem of determining sufficient conditions
for coherence in terms of the underlying rewriting system of a
structure that do not rely on either termination or
confluence. This very quickly leads one to consider two further
coherence questions: If there are diagrams that do not commute, then
is there at least a decision procedure that determines whether a given
diagram commutes? Is it at least true that for any finite collection
of functors and natural transformations, there is always a finite set
of diagrams whose commutativity implies the commutativity of all
diagrams built from this structure? 

This paper sets out to solve the various coherence questions by
vigourously pursuing the idea that two morphisms with the same
source and target in a free covariant structure on a discrete category
commute precisely when they admit a planar subdivision such that each
face is an instance of naturality, or of functoriality or of one of
the axioms. The guiding intuition behind this approach is that a span
that cannot be completed into a square can never appear in such a 
subdivision. 

We begin in Section \ref{sect:structures} by developing a framework
for viewing a two dimensional structure on a category as a term
rewriting system modulo a two dimensional congruence. In Section
\ref{sect:lambek}, we resolve the problem of determining sufficient
conditions for the existence of a decision procedure for the
commutativity of diagrams. We call this  problem the ``Lambek
coherence problem'', since it is inspired by Lambek's paper on closed
categories and deductive systems \cite{Lambek:coh1}. In Section
\ref{sect:maclane}, we determine sufficient conditions for all
diagrams to  commute, a problem that we call the ``Mac Lane coherence
problem''. As an immediate application, we construct an example of a
structure that has no  finite basis for Mac Lane coherence but is
otherwise well behaved. 
Finally, in Section \ref{sect:imc}, we apply the theory to iterated
monoidal categories and obtain a new and conceptual proof of their
coherence theorem.  

\section{$2$-Structures}\label{sect:structures}

The purpose of this section is to describe a two-dimensional covariant
structure on a category as a certain type of term rewriting system. At
the onset, we are presented with certain basic functors and natural 
transformations, together with an equational theory on the absolutely
free term algebra generated by the functors, as well as an
equational theory on the absolutely free reduction system generated by
the natural transformations. The idea of viewing such a system as a
term rewriting system can be found, for instance, in Meseguer's
Rewriting Logic \cite{Meseguer_rl}. An important point to note is that
Rewriting Logic does not allow any additional equations on reductions,
beyond those required to ensure naturality and functoriality. In other
words, it does not provide a facility for specifying coherence
conditions. We begin by describing the first layer of structure.

\begin{definition}
  Given a graded set of function symbols $\F := \sum_n \F_n$ and a set
  $X$, the absolutely free term algebra generated by $\F$ on $X$ is
  denoted by $\Term_X(\F)$. 
\end{definition}

The next layer of structure adds an equational theory to
$\Term_X(\F)$:

\begin{definition}
  Given a graded set of function symbols $\F$, a set $X$ and a set of
  equations $\theta_\F$ on $\Term_X(\F)$, we denote by
  $\Term_X(\F,\theta_\F)$ the quotient of $\Term_X(\F)$ by the
  smallest congruence generated by $\theta_\T$. We write $[t]$ for the
  image of a term $t$ under the homomorphism $\Term_X(\F) \to
  \Term_X(\F,\theta_\F)$. 
\end{definition}

The next layer of structure adds some reduction rules between
congruence classes of $\Term_X(\F,\theta_\F)$. 

\begin{definition}\label{defn:trt}
A \emph{labelled term rewriting theory} is a tuple $\langle X, \F,
\theta_\F, \mathcal{L}, \T \rangle$, where $X$ is a countably infinite
set of variables, $\F$ is a graded set of function symbols,
$\theta_\F$ is a system of $\Term_X(\F)$-equations, $\mathcal{L}$ is a
set of labels and $\T$ is a subset of $\mathcal{L} \times
(\Term_X(\F,\theta_\F))^2$ satisfying the following consistency 
conditions: 
\begin{center}
  If $(\alpha,s_1,t_1)$ and $(\alpha,s_2,t_2)$ are in $\T$ then
    $s_1 = s_2$ and $t_1 = t_2$.
\end{center}
If $(\alpha,s,t) \in \T$, we write $\alpha:\ell \to r$. A member
of $\T$ is called a \emph{labelled reduction rule}. 
\end{definition}

Given a labelled term rewriting theory $\langle X, \F, \theta_\F,
\mathcal{L}, \T \rangle$, the particular choice of $X$ and
$\mathcal{L}$ is irrelevant. What is important is simply that there
are sufficient variables and labels. Accordingly, we shall henceforth
suppress explicit mention of the variables and labels and write
$\langle \F, \theta_\F, \T \rangle$ for a labelled term rewriting
theory. A labelled term rewriting theory embodies the basic reductions
that are to generate all others. The next step is to obtain an
analogue of the absolutely free term algebra for this higher
dimensional layer of structure. This is achieved by the following
definition, we there notation $\overline{x}^n$ is an abbreviation for
$x_1,\dots,x_n$ and $F(\overline{s}^n/\overline{x}^n)$ denotes the
uniform substitution of the free variables $\overline{x}^n$ by
$\overline{s}^n$.

\begin{definition}
  Given a labelled term rewriting theory $\mathcal{R} := \langle \F,
  \theta_\F,\T\rangle$ and a category $\C$, the set of
  \emph{reductions generated by $\mathcal{R}$} is denoted
  $\Term_\C(\F,\theta_\F,\T)$ and is generated inductively by the
  following rules:  
      
  \medskip

  \begin{center}
    \begin{tabular}{cl}
      \AxiomC{}
      \UnaryInfC{$[f]:[s] \to [t]$}
      \DisplayProof
      
      &  {\rm (Inheritance)}

      \bigskip\\
      
      \AxiomC{$\varphi_1:[s_1] \to [t_1]~ \dots~ \varphi_n:[s_n] \to [t_n]$} 
      \UnaryInfC{$F(\varphi_1,\dots,\varphi_n):[F(s_1,\dots,s_n)] \to
        [F(t_1,\dots,t_n)]$} 
      \DisplayProof
      
      & {\rm (Structure)}

      \bigskip\\
      
      \AxiomC{$\alpha:[F(\overline{x}^n)] \to [G(\overline{x}^n)]$}
      \AxiomC{$(\varphi_i:[s_i] \to [t_i])_{i=1}^n$}
      \BinaryInfC{$\tau(\varphi_1,\dots,\varphi_n):[F(\overline{s}^n/
        \overline{x}^n)] \to [G(\overline{t}^n/\overline{x}^n)]$} 
      \DisplayProof
      
      & {\rm (Replacement)}

      \bigskip\\
      
      \AxiomC{$\varphi: [s] \to [u]$}
      \AxiomC{$\psi: [u] \to [t]$}
      \BinaryInfC{$(\varphi\of\psi):[s] \to [t]$}
      \DisplayProof
      
      &  {\rm (Transitivity)}

    \end{tabular}
  \end{center}
  In the {\rm (Inheritance)} rule, $f:s\to t$ is in $\Mor(\C)$. In the
  {\rm (Structure)} rule, $F$ is a function symbol of rank $n$. In the
  {\rm (Replacement)} rule. 
\end{definition}

\begin{example}\label{example:associative}
  
  Let $\C$ be the discrete category generated by the set
  $\{A,B,C,D\}$. Consider the term rewriting theory with a single
  binary function symbol $\tensor$, an empty equational theory on
  terms and the single reduction rule:
  \[ \alpha(x,y,z) : x\tensor(y\tensor z) \to (x\tensor y)\tensor z\]
  
  A derivation of $\arrow{A \tensor (B \tensor (C \tensor
    D))}{(A\tensor B)\tensor(C\tensor D)}$   
  in this system is given by:
    \begin{prooftree}
      \AxiomC{}
      \UnaryInfC{$\arrow{1_A~:~A}{A}$}
      \AxiomC{}
      \UnaryInfC{$\arrow{1_B~:~B}{B}$}
      \AxiomC{}
      \UnaryInfC{$\arrow{1_C~:~C}{C}$}
      \AxiomC{}
      \UnaryInfC{$\arrow{1_D~:~D}{D}$}
      \TrinaryInfC{$\arrow{\alpha(1_B,1_C,1_D)~:~B\tensor(C\tensor D)}
        {(B\tensor C)\tensor D}$}
      \BinaryInfC{
        $\arrow{1_A\tensor\alpha(1_B,1_C,1_D)~:~A\tensor(B\tensor(C\tensor D))} 
        {(A\tensor B) \tensor (C\tensor D)}$
      }
    \end{prooftree}
\end{example}

The consistency condition in Definition \ref{defn:trt} easily yields
the following lemma, which asserts that we may equate reductions with
their labels. 

\begin{lemma}
  Let $\C$ be a category and $\langle \F, \theta_\F, \T\rangle$ be a labelled
  term rewriting theory. Then:
  \begin{enumerate}
    \item  If $\alpha:s \to t$ and $\alpha:s' \to t'$ are in
    $\Term_\C(\F,\theta_\F,\T)$, then $s = s'$ and $t = t'$.
    \item For $t \in \Term_{\Ob(C)}(\F,\theta_\F)$, there is a unique
      identity reduction $1_t:[t] \to [t]$ in
      $\Term_\C(\F,\theta_\F,\T)$ given inductively by:
      \begin{equation*}
        1_t = 
        \begin{cases}
          [1_t] & \text{if $t \in \Ob(\C)$,}\\
          F(1_{t_1},\dots,1_{t_n}) & \text{if $t = F(t_1,\dots,t_n)$}
        \end{cases}
      \end{equation*}
    \end{enumerate}
\end{lemma}

We now have the main ingredients for defining a covariant structure
carried by a category. What remains is to ensure that the function
symbols behave as functors, that the reduction rules behave as natural
transformations and that we can stipulate coherence conditions. 

\begin{definition}
  Let $\C$ be a category. A \emph{covariant $2$-structure} on $\C$ is a
  tuple $\langle\F,\theta_\F,\T,\theta_\T \rangle$, where
  $\langle \F,\theta_\F,\T\rangle$ is a labelled 
  term rewriting theory and $\theta_\T$ is a set of equations
  on $\Term_\C(\F,\theta_\F,\T)$ satisfying the following consistency
  condition:
  \begin{center}
    If $\varphi_1 = \varphi_2$ is in $\theta_\T$ and $\varphi_1: s_1 \to t_1$
    and $\varphi_2:s_2 \to t_2$, then $s_1 = s_2$ and $t_1 = t_2$.
  \end{center}
  In other words, we can set two reductions to be equal only if their
  source and target match. We further stipulate that the following
  equations form a subset of $\theta_\T$. 
  \begin{center}
    \begin{tabular}{cl}
      $1_s\of\varphi = \varphi$ & {\rm (ID 1)}\\
      $\varphi\of 1_t = \varphi$ & {\rm (ID 2)}\\
      $\varphi\of(\psi\of\rho) = (\varphi\of\psi)\of\rho$ & {\rm (Assoc)}\\
      $F(\varphi_1,\dots,\varphi_n)\of F(\psi_1,\dots,\psi_n) =
      F(\varphi_1\of\psi_1,\dots,\varphi_n\of\psi_n)$ & {\rm (Funct)}\\
      $\varphi(\varphi_1,\dots,\varphi_n) =
      s(\varphi_1,\dots,\varphi_n)\of\varphi(t_1,\dots,t_n)$ & {\rm (Nat
        1)}\\
      $\varphi(\varphi_1,\dots,\varphi_n) = \varphi(s_1,\dots,s_n)\of
      t(\varphi_1,\dots,\varphi_n)$ & {\rm (Nat 2)}
    \end{tabular}
  \end{center}
  In the above, $\varphi:s \to t$ and $\varphi_i:s_i \to t_i$ are in
  $\Term_\C(\F,\theta_\F,\T)$ and $F \in \F_n$. 
\end{definition}

Since the only structures we deal with in this paper are covariant, we
shall henceforth take ``$2$-structure'' to mean ``covariant
$2$-structure''. Our final task is to generate a congruence on
reductions.

\begin{definition}
  If $\langle\tuple\rangle$ is a $2$-structure on a category $\C$, then 
  $\widehat{\theta_\T}$ denotes the smallest congruence generated by
  $\theta_\T$ on $\Term_\C(\F,\theta_\F,\T)$. It is generated
  inductively by the following rules:

  \begin{center}
    \begin{tabular}{cll}
    
      \bigskip
      \AxiomC{}
      \UnaryInfC{$\varphi=\varphi$}
      \DisplayProof
      
      & {\rm (Identity)}
      
      & $\varphi \in \T$\\
      \bigskip
      
      \AxiomC{}
      \UnaryInfC{$\varphi_1 = \varphi_2$}
      \DisplayProof
      
      & {\rm (Inheritance)}

      & $(\varphi_1,\varphi_2) \in \theta_\T$\\
      \bigskip
      
      \AxiomC{$\varphi = \psi$}
      \UnaryInfC{$\psi = \varphi$}
      \DisplayProof

      &  {\rm (Symmetry)}\\
      
      \AxiomC{$\varphi_1=\psi_1~ \dots~ \varphi_n = \psi_n$} 
      \UnaryInfC{$F(\varphi_1,\dots,\varphi_n) = F(\psi_1,\dots,\psi_n)$} 
      \DisplayProof
    
      &  {\rm (Structure)}
      \bigskip
      & $F \in \F_n$\\

      \AxiomC{$\varphi_1=\psi_1~ \dots~ \varphi_n = \psi_n$}
      \UnaryInfC{$\tau(\varphi_1,\dots,\varphi_n) =
        \tau(\psi_1,\dots,\psi_n)$}  
      \DisplayProof
      
      & {\rm (Replacement)}
      \bigskip
      & $\tau \in \T_n$\\

      \AxiomC{$(\varphi_1 = \psi_1): s \to u$}
      \AxiomC{$(\varphi_2=\psi_2): u \to t$}
      \BinaryInfC{$(\varphi_1\of\psi_1 = \varphi_2\of\psi_2) :s \to t$} 
      \DisplayProof
      
      &  {\rm (Transitivity)}
      
    \end{tabular}

  \end{center}
\end{definition}

We are now in a position to define our main object of study.

\begin{definition}
Given a $2$-structure $\langle\tuple\rangle$ on a category $\C$, we
use $\Fr_\C(\tuple)$ to denote the quotient 
$\Term_\C(\F,\theta_\F,\T)/\widehat{\theta_\T}$.  
\end{definition}

Our notion of a covariant $2$-structure essentially recasts Kelly's
definition of a fully covariant club \cite{Kelly_coherence} in the
language of term rewriting theory. The construction of $\Fr_\C(\S)$
parallels Kelly's construction of the functor part of an equational
doctrine on ${\bf Cat}$ whose algebras are precisely the free
$\S$-algebras, relative to an appropriate notion of weak morphism
between $\S$-algebras. With this observation, we have the following
theorem.

\begin{theorem}[Kelly, \cite{Kelly_clubs}]
  $\Fr_\C(\S)$ is the initial $\S$-algebra on $\C$. \qed
\end{theorem}
Our main concern is to fully describe $\Fr_\C(\S)$ in the case where
$\C$ is a discrete category in terms of the generators and relations
in $\S$.  Moreover, we only wish to consider diagrams that are as
general as possible. To this end, we formalise the notion that a
reduction has ``as many variables as possible''. We begin by defining
the shape of a reduction.

\begin{definition}
  Let $\langle\tuple\rangle$ be a $2$-structure on a category
  $\C$. The \emph{Shape} of a reduction $\alpha \in
  \Term_\C(\F,\theta_\F,\T)$ is defined recursively by the following:
  \begin{equation*}
    \Shape(\alpha) = 
    \begin{cases}
      \Shape(\alpha_1)\of\Shape(\alpha_2) & \text{if $\alpha =
        \alpha_1\of\alpha_2$} \\
      \tau(\Shape(\alpha_1),\dots,\Shape(\alpha_n)) & \text{if $\alpha
        = \tau(\alpha_1,\dots,\alpha_n)$}\\
      F(\Shape(\alpha_1),\dots,\Shape(\alpha_n)) & \text{if $\alpha =
        F(\alpha_1,\dots,\alpha_n)$}\\
      \circ & \text{otherwise}
    \end{cases}
  \end{equation*}
\end{definition}

In the system from Example \ref{example:associative}, we have:

\begin{equation*}
  \Shape(\alpha(1_A,1_B,1_C)) = \Shape(\alpha(1_A,1_A,1_A)) =
  \alpha(\circ,\circ,\circ)
\end{equation*}

We now need a precise definition of the variables present in a
reduction.

\begin{definition}
  Given a $2$-structure $\langle\tuple\rangle$ on a category $\C$, the
  set of variables in a reduction $\alpha \in
  \Term_\C(\F,\theta_\F,\T)$ is defined recursively as follows:
  \begin{equation*}
    \var(\alpha) =
    \begin{cases}
      \var(\alpha_1)\cup\var(\alpha_2) &\text{if $\alpha =
        \alpha_1\of\alpha_2$}\\
      \bigcup_{i=1}^n \var(\alpha_i) & \text{if $\alpha =
        \tau(\alpha_1,\dots,\alpha_n)$}\\
      \bigcup_{i=1}^n \var(\alpha_i) & \text{if $\alpha =
        F(\alpha_1,\dots,\alpha_n)$}\\
      \alpha & \text{otherwise}
    \end{cases}
  \end{equation*}
\end{definition}

Returning to Example \ref{example:associative}, we find that
$\var(\alpha(1_A,1_B,1_C)) = \{1_A,1_B,1_C\}$, whereas
$\var(\alpha(1_A,1_A,1_A)) = \{1_A\}$. We can finally nail down what
we mean when we say a reduction has the maximum possible number of
variables. 

\begin{definition}
  Given a $2$-structure $\langle\tuple\rangle$ on a category $\C$, a
  reduction $\alpha\in \Term_\C(\F,\theta_\F,\T)$ is \emph{in general
    position} if 
  \begin{equation*}
    |\var(\alpha)| = \max\{|\var(\tau)| ~:~ \tau \in
    \Term_\C(\F,\theta_\F,\T) ~ \textrm{and}~ \Shape(\tau) =
    \Shape(\alpha)\}.
  \end{equation*}
\end{definition}

\begin{example}
  Consider the system from Example \ref{example:associative} augmented
  with the following reduction rule:
  \[
  \beta(x) : x\tensor x \to x
  \]

  Then, 
  $$\alpha(1_A,1_A,1_B)\of(\beta(1_A)\tensor 1_B): A \tensor (A
  \tensor B) \to A \tensor B$$ 
  is in general position, whereas
  $$\alpha(1_A,1_A,1_B)~:~A\tensor(A\tensor B) \to (A\tensor
  A)\tensor B$$ 
  is not in general position. 
\end{example}

Refining our previous remarks, in order to investigate coherence
problems, we need only consider reductions that are in general
position in a $2$-structure on a discrete category. In the following
section, we tackle the problem of deciding whether such a diagram
commutes.   

\section{Lambek Coherence}\label{sect:lambek}

Given a $2$-structure $\S$ on a category $\C$,
we often wish to determine whether a diagram in $\Fr_\C(\S)$
commutes. Such a diagram may commute due to commutativity of diagrams
already present in $\C$, or it may commute purely as a result of the
structure present in $\S$. It is the latter case that concerns us
here and, as such, we may make the assumption that $\C$ is
discrete.

\begin{definition}[Lambek Coherence]
  A $2$-structure on a category $\C$ is \emph{Lambek coherent} if
  it is decidable whether two reductions in general position having
  the same source and target are equal whenever $\C$ is a discrete
  category. 
\end{definition}

An immediate question that arises is whether every $2$-structure is
Lambek coherent. Unsurprisingly, the answer is no, even in the case
that the $2$-structure is finitely presented. 

\begin{theorem}\label{thm:undecidable}
  There exist finitely presented $2$-structures that are not
  Lambek-coherent. 
\end{theorem}
\begin{proof}
  Let $\langle X |R\rangle$ be a finite presentation for a monoid with an
  unsolvable word problem. Let $\S$ be the structure consisting of
  a single unary function symbol $F$, reductions $\tau_i : F(x) \to
  F(x)$ for every $\tau_i \in X$ and relations $(\omega_i,\omega_j)$ for
  every $(\omega_i,\omega_j) \in R$. Then the Lambek coherence problem for
  $\S$ is equivalent to the word  problem for $\langle X |R\rangle$
  and is hence undecidable. 
\end{proof}

Seeking to understand the reasons why a $2$-structure could fail to be
Lambek coherent, one may well suspect that termination is a key
feature. 

\begin{definition}[Termination]
A $2$-structure on a category $\C$ is \emph{terminating} if whenever
$\C$ is a discrete category, every infinite chain
\[
t_1 \stackrel{\alpha_1}{\to} t_2 \stackrel{\alpha_2}{\to} t_3
\stackrel{\alpha_3}{\to} \dots
\]
in $\Fr_\C(\S)$ contains cofinitely many identity reductions. 
\end{definition}

One may reasonably put forward the question as to whether every
\emph{terminating} $2$-structure is Lambek coherent. It is a classical
result of term rewriting theory that termination is an undecidable
property (see, for example, \cite{terese:decidability}). Since the
example constructed in Theorem \ref{thm:undecidable} is not
terminating, it is entirely possible that this is the point at which
undecidability of Lambek coherence creeps in. In this section, we show
that this intuition is roughly correct. In fact, we require a slightly
weaker property than termination, which allows the result to be
applicable to systems such as that consisting solely of a unary
function symbol $F$ and the reduction rule $F(x) \to
F(F(x))$. However, we do need to work modulo the decidability of the
word problem at the object level. 

Our general approach is to examine the collection of subdivisions of a
given parallel pair of reductions in general position. We seek a
general criterion that ensures that any such pair admits only finitely
many subdivisions. If this is the case, we may enumerate the
subdivisions of a given parallel pair and examine each face for
commutativity. We first need to develop an appropriate definition of a
subdivision. 

\subsection{Subdivisions}

A subdivision of a parallel pair of reductions is, in the first
instance, a collection of reductions having the same source and
target.

\begin{definition}
  An \emph{st-graph} is a labelled directed graph $G$ (possibly with
  loops and multiple edges) together with
  two distinguished vertices $u$ and $v$, called the source and target
  of $G$ respectively, such that for any other
  vertex $w \in G$, there exist paths $u \to w$ and $w \to v$ in $G$.
\end{definition}

Of particular interest to us are st-graphs contained in the reduction
graph of a $2$-structure.

\begin{definition}
A morphism  $\varphi \in \Fr_\C(\M)$ is \emph{irreducible} if $\varphi =
\varphi_1\of\varphi_2$ implies that $\varphi_1 = 1$ or $\varphi_2 = 1$.
\end{definition}

\begin{definition}[Reduction graph]
  Let $\S$ be a $2$-structure on a discrete category $\C$.  The
  expression   $\Red_{\S,\C}$ denotes the \emph{reduction graph} of
  $\S$ on   $\C$. This graph has 
  \begin{itemize}
    \item Vertices: The set $\Term_{\Ob(\C)}(\F,\theta_\F)$.
    \item Edges: Irreducible morphisms in $\Fr_\C(\S)$. 
  \end{itemize}
\end{definition}

A subdivision corresponds to a particular way of embedding an st-graph
in the oriented plane. Given a graph $G$, we use $|G|$ to denote its
geometric realisation. We write $\R^2$ for the plane with the
clockwise orientation.

\begin{definition}
  Let $G$ be a graph and $\alpha,\beta \in G(s,t)$.. A
  \emph{pre-subdivision} of $\langle\alpha,\beta\rangle$ is a pair
  $(S,\varphi)$ such that:
  \begin{enumerate}
    \item $S$ is an st-graph.
    \item $\{\alpha,\beta\} \subseteq S \subseteq G$.
    \item $\varphi:|S| \hookrightarrow \R^2$ is a planar embedding.
    \item For every edge $\gamma \in S$, the image $\varphi(|\gamma|)$
      is contained in the region of $\R^2$ bounded by
      $\varphi(|\alpha|)$ and $\varphi(|\beta|)$.
  \end{enumerate}
  We use $\mathrm{PSub}_G(\alpha,\beta)$ to denote the set of all
  pre-subdivisions of $\langle\alpha,\beta\rangle$ in $G$.
\end{definition}

The definition of pre-subdivisions admits too many different embeddings
of the same graph. To this end, we define a useful equivalence
relation on pre-subdivisions. 

Given a graph $G$, let $\alpha,\beta \in G(s,t)$. Let
$\langle S_1,\varphi\rangle$ and $\langle S_2,\psi\rangle$ be
pre-subdivisions of $\langle\alpha,\beta\rangle$. Define $\sim$ to be
the equivalence relation on $\mathrm{PSub}_G(\alpha,\beta)$ generated
by setting $\langle S_1,\varphi\rangle \sim \langle S_2,\psi\rangle$
if:
\begin{enumerate}
  \item $S_1 = S_2$.
  \item $\varphi$ and $\psi$ are ambiently isotopic.
\end{enumerate}

We write $\Sub_G(s,t)$ for the quotient $\mathrm{PSub}_G(s,t)/\!\sim$.

\begin{definition}
  For a directed graph $G$ and $\alpha,\beta \in G(s,t)$, a
  \emph{subdivision} of $\langle\alpha,\beta\rangle$ is a member of
  $\Sub_G(s,t)$. For a $2$-structure $\S$ on a discrete category $\C$,
  a subdivision of a parallel pair of morphisms $\alpha,\beta \in
  \Fr_\C(S)$ is a subdivision of $\langle\alpha,\beta\rangle$ in
  $\Red_{\S,\C}$. The set of all such subdivisions is denoted
  $\Sub_{\S,\C}(\alpha,\beta)$.
\end{definition}

Recall that a directed graph $G$ is \emph{locally finite} if $G(s,t)$
is finite for all vertices $s,t \in G$. The following sequence of
lemmas establishes a correspondence between local finiteness and
finitely many subdivisions. 

\begin{lemma}\label{lem:fin_image}
 For a directed graph $G$ and a finite planar subgraph $S \le
 G(s,t)$, there are only finitely many subdivisions of
 $\alpha,\beta \in G(s,t)$ having graph $S$.
\end{lemma}
\begin{proof}
  Since we only consider embeddings of $S$ up to ambient isotopy, a
  subdivision with graph $S$ is completely determined by the set of
  edges mapped to the region bounded by $\varphi(|\gamma_1|)$ and
  $\varphi(|\gamma_2|)$ for every 
  parallel pair of paths $\gamma_1,\gamma_2 \in S$. Since $S$ is
  finite, there are only finitely many possibilities for this.
\end{proof}

\begin{lemma}\label{lem:subgraph}
  An st-graph with source $s$ and target $t$ is finite if and only if
  it has finitely many planar st-subgraphs with source $s$ and target $t$.
\end{lemma}
\begin{proof}
($\Rightarrow$) A finite graph has finitely many subgraphs, so it
certainly has finitely many planar subgraphs.

($\Leftarrow$) Suppose that $G$ is an infinite st-graph with source $s$
and target $t$. Each path from $s$ to $t$ in $G$ determines a planar
subgraph of $G$, hence $G$ has infinitely many planar subgraphs with
source $s$ and target $t$.
\end{proof}

Combining the Lemma  \ref{lem:fin_image} and Lemma \ref{lem:subgraph},
we obtain the desired correspondence.

\begin{lemma}\label{lem:fin_sub}
If $G$ is a directed graph containing vertices $s$ and $t$, then
$G(s,t)$ is finite if and only if $\Sub_G(\alpha,\beta)$ is finite for
all $\alpha,\beta \in G(s,t)$
\end{lemma}

\subsection{Ensuring local finiteness}

By Lemma \ref{lem:fin_sub}, in order to ensure that every parallel
pair of paths in a directed graph has finitely many subdivisions, we
need only establish that the graph is locally finite. To this end, we
make the following definition.

\begin{definition}
  Let $G$ be a directed graph. A \emph{quasicycle} in $G$ is a pair
  $(T,t)$ such that:
  \begin{enumerate}
    \item $T$ is an infinite chain $t_0 \to t_1 \to \dots$ in $G$ 
    \item $t$ is a vertex in $G$.
    \item $G$ contains a path $t_i \to t$ for all $i \in \mathbb{N}$. 
  \end{enumerate}
\end{definition}

\begin{figure}[h]
\[
\vcenter{
\xymatrix{
  {t_0} \ar[r] \ar@/_1pc/[ddrr] & {t_1} \ar[r] \ar@/_/[ddr] & 
  {t_2} \ar[r] \ar[dd] & {t_3} \ar[r] \ar@/^/[ddl] &
  {t_4} \ar[r] \ar@/^1pc/[ddll]  & {\dots}\\
  \\
  && {t}
}
}
\]
\caption{A quasicycle}\label{fig:quasicycle}
\end{figure}
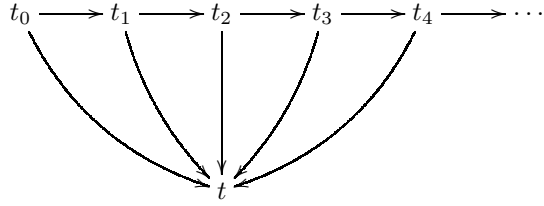

Quasicycles earn their name by being a slightly weaker notion than a
cycle. Figure \ref{fig:quasicycle} gives an example of a quasicycle
that is not a cycle. On the other hand, we have the following easy
result.

\begin{lemma}\label{lem:cycle}
  Let $C$ be a directed cycle and $c$ be a vertex in $C$. Then,
  $(C,c)$ is a quasicycle. 
\end{lemma}

For a directed graph $G$ and a vertex $s \in G$, we use $\Out_G(s)$ to
denote the set $\{t \in V(G) : G \textrm{ contains an edge } s\to
t\}$. We say that $G$ is \emph{finitely branching} if $\Out_G(s)$ is
finite for all vertices $s \in G$. One of our main technical tools is
the following graphical version of K\"{o}nig's Tree Lemma.

\begin{lemma}\label{lem:quasi}
  A finitely branching directed graph is locally finite if and only if
  it contains no quasicycles.
\end{lemma}
\begin{proof}
  Let $G$ be a labelled finitely branching directed graph.

  ($\Rightarrow$) Suppose that $G$
  contains a quasicycle $(T,t)$, where $T = t_0
  \stackrel{\alpha_0}{\to} t_1 \stackrel{\alpha_1}\to \dots$. If $t_i
  = t$ for some $i \in \N$ then $G(t_i,t_j)$ is infinite for all $j >
  i$. So, suppose that $t_i \ne t$ for all $i \in \N$. Since $t_i \to
  t$ for all $i \in \N$, there must be infinitely many pairs
  $(i,\beta_i)$, where $i \in \N$ and $\beta_i:t_i\to t$ is a path that
  does not factor through $t_j$ for any $j > i$. So, $G(t_0,t)$ is
  infinite.

  ($\Leftarrow$) Suppose that
  $G(s,t)$ is infinite. Since $\Out_G(s)$ is finite, it follows from the
  pigeon hole principle that there must exist some $s_0 \in \Out_G(s)$
  and an edge $\alpha_0:s \to s_0$ such that $G(s_0,t)$ is
  infinite. Continuing recursively, we obtain an infinite chain $s
  \stackrel{\alpha_0}{\to} s_0 \stackrel{\alpha_1}{\to} s_1
  \stackrel{\alpha_2}{\to} \dots$ such that $G$ contains  a path $s_i
  \to t$ for all $i \in \N$. So, $G$ contains a quasicycle. 
\end{proof}

\begin{definition}
  A $2$-structure $\S$ on a category $\C$ is \emph{quasicycle-free} if
  every quasicycle in  $\Red_{\S,\C}$ contains cofinitely many
  identity reductions. It is locally finite if   $\Red_{\S,\C}$ is
  locally finite.  
\end{definition}

Lemma \ref{lem:quasi} very quickly yields the following fundamental
result.

\begin{proposition}\label{prop:quasi}
  A finitely presented $2$-structure on a discrete category is locally
  finite if and only if it is quasicycle-free.
\end{proposition}
\begin{proof}
  Let $\S$ be a finitely presented $2$-structure on a discrete category
  $\C$. Since each term has finitely many subterms and $\S$ has
  finitely many reduction rules, $\Red_{\S,\C}$ is finitely branching. Lemma
  \ref{lem:quasi} then applies. 
\end{proof}

Lemma \ref{lem:fin_sub} and Proposition \ref{prop:quasi} together
imply that a finitely presented quasicycle-free $2$-structure on a
discrete category has only finitely many subdivisions for every
parallel pair of reductions. A ready supply of such $2$-structures is
provided by the following lemma.

\begin{lemma}
  A terminating $2$-structure on a discrete category is quasicycle free.
\end{lemma}

By Lemma \ref{lem:cycle}, a quasicycle-free directed graph is
acyclic. The following theorem establishes that every face of a
subdivision in an acyclic graph is itself a parallel pair of paths. It
was originally discovered by Power \cite{Power:pasting2} in his
investigation of pasting diagrams in $2$-categories.

\begin{theorem}[Power \cite{Power:pasting2}]\label{thm:power}
  A planar $st$-graph is acyclic if and only if every face has a
  unique source and target.
\end{theorem}

Theorem \ref{thm:power} allows us to very easily deduce the following
result. 

\begin{proposition}\label{prop:faces}
  Let $\S$ be a $2$-structure on a discrete category $\C$ and $\alpha,\beta
  \in \Red_{\S,\C}(s,t)$. Then, the following statements are equivalent:
  \begin{enumerate}
    \item $\alpha = \beta$ in $\Fr_\C(\S)$. 
    \item There is a subdivision of $\langle\alpha,\beta\rangle$ in
      $\Red_{\S,\C}(s,t)$ such that each face commutes in
      $\Fr_\C(\S)$.
    \item There is a subdivision of $\langle\alpha,\beta\rangle$ in
      $\Red_{\S,\C}(s,t)$ such that each face is either an instance of
      functoriality, or an instance of naturality or an instance of
      one of the equations in $\theta_\T$.
  \end{enumerate}
\end{proposition}

\subsection{The Lambek coherence theorem}

With Proposition \ref{prop:faces}, Proposition \ref{prop:quasi} and
Lemma \ref{lem:fin_sub}, we are seemingly home and dry since we now
know that every quasicycle-free $2$-structure on a discrete category
has only finitely many subdivisions for each parallel pair and we can
just check every face to see whether it is an instance of
functoriality, naturality or a coherence axiom. There is, however, one
catch - we may not be able to decide whether a given face is an
instance of an axiom! 

\begin{definition}[Unification]
  Let $\F$ be a ranked set of function symbols on a set $X$ and
  $\theta_\F$ be an equational theory on $\Term_X(\F)$. A
  $\theta_\F$-unification problem is a finite set:
  \[
  \Gamma = \{(s_1,t_1),\dots, (s_n,t_n)\},
  \]
  where for $1 \le i \le n$, we have that $s_i$ and $t_i$ are in
  $\Term_\F(X)$. A \emph{unifier} for  $\Gamma$  a homomorphism
  $\sigma: \Term_\F(X) \to \Term_\F(X)$ such 
  that $\sigma(s_i) =_{\theta_\F} \sigma(t_i)$ for all $1 \le i \le
  n$. The set $\Gamma$ is \emph{unifiable} if it admits at least one
  unifier. 
\end{definition}

Unification theory is an important technical component of automated
reasoning and logic programming, as it provides a means for testing
whether two sequences of terms are syntactic variants of each other. A
good survey of the field is provided by
\cite{Baader:unification}. In the case where the theory $\theta_\F$ is
empty, the unification problem is easily shown to be
decidable. Unfortunately, the equational unification problem is in
general undecidable.  

\begin{definition}
  A $2$-structure $\langle\tuple\rangle$ has \emph{decidable term
    unification} if $\langle \F, \theta_\F\rangle$ has a decidable
  unification problem. 
\end{definition}

We can finally establish the main theorem of this section.

\begin{theorem}[Lambek Coherence]\label{thm:lambek}
  A finitely presented quasicycle-free structure with decidable term
  unification on a discrete  category is Lambek Coherent. 
\end{theorem}
\begin{proof}
  Let $\S$ be a $2$-structure on a discrete category $\C$ satisfying
  the hypotheses and $\alpha,\beta \in \Red_{\S\C}(s,t)$. By
  Proposition \ref{prop:quasi} and Lemma \ref{lem:fin_sub}, we can
  enumerate the subdivisions of $\langle\alpha,\beta\rangle$. Since
  each subdivision has only finitely many faces, we may apply
  Proposition \ref{prop:faces}  to determine
  whether every face of a subdivision commutes in $\Fr_\C(\S)$, since
  $\S$ has decidable term unification. 
\end{proof}

\begin{corollary}
  It is undecidable whether a finitely presented discrete structure with
  decidable term unification is quasicycle-free.
\end{corollary}
\begin{proof}
  The discrete structure constructed in the proof of Theorem
  \ref{thm:undecidable} clearly has an empty equational theory on
  terms and so has decidable term unification. It
  follows from Theorem \ref{thm:lambek} that, were we able to
  determine whether the structure is quasicycle free, then we would
  be able to decide whether a finite monoid presentation has a
  decidable word problem. 
\end{proof}

As a particular application of Theorem \ref{thm:lambek}, any
terminating $2$-structure with an empty equational theory is Lambek
coherent. This includes, amongst others, categories with a directed
associativity  \cite{laplaza:associative}. The unification problem for
an associative binary symbol is decidable
\cite{Baader:unification}. It follows then, from Theorem
\ref{thm:lambek} that the following $2$-structures are Lambek
coherent (in each case we need only check that the $2$-structure is
terminating): 
\begin{itemize}
  \item Distributive categories with strict associativities and strict
    units \cite {laplaza:distributive}
  \item Weakly distributive categories with strict associativity and strict units
    \cite{CockettSeely:wdc}.
\end{itemize}

An example of a non-terminating $2$-structure that is Lambek-coherent
is provided by the system $\F(X) \to F(F(X))$, since this is easily
seen to be quasicycle free. 

In the following section, we continue our investigation of quasicycle
free $2$-structures and derive sufficient conditions for such a system
to be Mac Lane coherent.

\section{Mac Lane Coherence}\label{sect:maclane}

The last section was concerned with deciding whether a given pair of
parallel morphisms is equal or, equivalently, whether a given diagram
in general position commutes. In this section, we tackle the problem
of determining sufficient conditions for all such diagrams to commute.

\begin{definition}
  Let $\S$ be a $2$-structure on a discrete category $\C$. We say that
  $\S$ is \emph{Mac Lane coherent} if every pair of morphisms
  in general position in $\Fr_\C(\S)$ with the same source and target
  are equal. 
\end{definition}

Our rough goal in this section is to find a minimal set of diagrams in
general position whose commutativity implies the commutativity of all
other such diagrams in $\Fr_\C(\S)$ for some $2$-structure $\S$ on a
discrete category $\C$. To this end, we define what it means for one
subdivision to be finer than another. The idea driving idea is that we
only wish to consider those subdivisions that do not embed into a
finer subdivision.

\begin{definition}
  Let $\S$ be a $2$-structure on a discrete category $\C$, and
  $\alpha,\beta \in \Red_{\S,\C}(s,t)$ and $(S_1,\varphi),(S_2,\psi)
  \in \Sub_{\S,\C}(\alpha,\beta)$. We say that $(S_1,\varphi)$ is
  \emph{coarser} than $(S_2,\psi)$ if there is a graph embedding
  $\Lambda:S_1\to S_2$ making the following diagram
  commute. In this case, we also say that $(S_2,\psi)$ is \emph{finer}
  than $(S_1,\varphi)$ and we write $(S_1,\varphi) \preceq
  (S_2,\psi).$ 
  \[
  \xymatrix@R=-0.1pc{
    {S_1} \ar[r]^{|\cdot |} \ar[dd]_{\Lambda}& {|S_1|} \ar@/^/[dr]^{\varphi}
    \ar[dd]_{|\Lambda|}\\
    &&{\R^2}\\
    {S_2} \ar[r]_{|\cdot |} &{|S_2|} \ar@/_/[ur]_{\psi}
  }
  \]
  We define the refinement order to be the antisymmetric closure of
  $\preceq$. 
\end{definition}

We shall abuse notation slightly and henceforth write $\preceq$ for
the refinement order. It is immediate from the definitions that the
set of subdivisions of a parallel pair of morphisms forms a poset
under refinement.

\begin{definition}
  Let $\S$ be a $2$-structure on a discrete category $\C$ and
  $\alpha,\beta \in \Red_{\S,\C}(s,t)$. A \emph{maximal subdivision} of
  $\langle\alpha,\beta\rangle$ is a maximal element of
  $(\Sub_{\S,\C}(\alpha,\beta),\preceq)$. 
\end{definition}

The idea behind the definition of a maximal subdivision is that these
are precisely the ones which cannot be further subdivided. Indeed, we
have the following lemma.

\begin{lemma}\label{lem:maximalface}
  A finitely presented quasicycle free $2$-structure on a discrete
  category is Mac Lane coherent if and only if every parallel pair of
  reductions in general position admits a maximal subdivision, each
  face of which commutes. 
\end{lemma}
\begin{proof}
  The direction ($\Leftarrow$) is trivial. For the other direction,
  let $\S$ be a quasicycle-free $2$-structure on a discrete category 
  $\C$. Let $\alpha,\beta \in \Red_{\S,\C}(s,t)$. Since $\S$
  is quasicycle-free, it follows from Proposition \ref{prop:quasi} and
  Lemma \ref{lem:fin_sub} that $\Sub_{\S,\C}(\alpha,\beta)$ is
  finite. Therefore, $\langle\alpha,\beta\rangle$ admits a maximal
  subdivision $(S,\varphi)$. By Theorem \ref{thm:power}, every face of
  $(S,\varphi)$ has a unique source and target. Since $\S$ is Mac Lane
  coherent, each of these faces commutes.
\end{proof}

In order to make Lemma \ref{lem:maximalface} effective, we need to
characterise those parallel pairs of morphisms that can occur as faces
of a maximal subdivision.

\begin{definition}[Zig-zag subdivision]
Let $G$ be a directed graph and $\alpha,\beta \in G(s,t)$. Suppose
that  
\begin{align*}
  \alpha &= s \stackrel{\alpha_0}{\to}a_0\stackrel{\alpha_1}{\to}\dots
  \stackrel{\alpha_{n-1}}{\to}a_{n-1}\stackrel{\alpha_n}{\to}t\\ 
  \beta &=   s \stackrel{\beta_0}{\to}b_0\stackrel{\beta_1}{\to}\dots
  \stackrel{\beta_{m-1}}{\to}b_{m-1}\stackrel{\beta_m}{\to}t
\end{align*}
and that each $\alpha_i$ and $\beta_i$ is irreducible. Let $U$ be the
forgetful functor from directed graphs to graphs that forgets the
direction of edges. A
\emph{zig-zag subdivision} of $\langle\alpha,\beta\rangle$ is a
subdivision $(S,\varphi)$ of $\langle\alpha,\beta\rangle$ such that
$U(S)$ contains a path from $U(a_i)$ to $U(b_j)$ for some pair
$(i,j)$, with $0 \le i \le n-1$ and $0 \le j \le m-1$. We call the
preimage of this path \emph{the zig-zag of $S$}.
\end{definition}

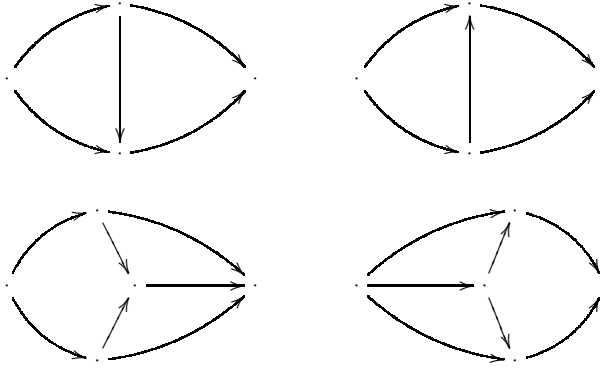
\begin{figure}[h]
  \begin{tabular}{ll}
  $
  \def\objectstyle{\scriptstyle}
  \def\labelstyle{\scriptstyle}
  \begin{xy}
    (0,0)*+{\cdot}="a";
    (15,10)*+{\cdot}="b";
    (15,-10)*+{\cdot}="c";
    (33,0)*+{\cdot}="d";
    {\ar@{->}@/_/ "c";"d"};
    {\ar@{->}@/^/ "a"; "b"};
    {\ar@{->}@/_/ "a"; "c"};
    {\ar@{->}@/^/ "b";"d"};
    {\ar@{->} "b";"c"}
  \end{xy}
  $
  & \qquad
  $
  \def\objectstyle{\scriptstyle}
  \def\labelstyle{\scriptstyle}
  \begin{xy}
    (0,0)*+{\cdot}="a";
    (15,10)*+{\cdot}="b";
    (15,-10)*+{\cdot}="c";
    (33,0)*+{\cdot}="d";
    {\ar@{->}@/_/ "c";"d"};
    {\ar@{->}@/^/ "a"; "b"};
    {\ar@{->}@/_/ "a"; "c"};
    {\ar@{->}@/^/ "b";"d"};
    {\ar@{->} "c";"b"}
  \end{xy}
  $
  \\
  \\
  $
  \def\objectstyle{\scriptstyle}
  \def\labelstyle{\scriptstyle}
  \begin{xy}
    (0,0)*+{\cdot}="a";
    (12,10)*+{\cdot}="b";
    (12,-10)*+{\cdot}="c";
    (33,0)*+{\cdot}="d";
    (17,0)*+{\cdot}="e";
    {\ar@{->}@/_/ "c";"d"};
    {\ar@{->}@/^/ "a"; "b"};
    {\ar@{->}@/_/ "a"; "c"};
    {\ar@{->}@/^/ "b";"d"};
    {\ar@{->} "e"; "d"};
    {\ar@{->} "b";"e"};
    {\ar@{->} "c";"e"};
  \end{xy}
  $
  & \qquad
  $
  \def\objectstyle{\scriptstyle}
  \def\labelstyle{\scriptstyle}
  \begin{xy}
    (0,0)*+{\cdot}="a";
    (21,10)*+{\cdot}="b";
    (21,-10)*+{\cdot}="c";
    (33,0)*+{\cdot}="d";
    (17,0)*+{\cdot}="e";
    {\ar@{->}@/_/ "c";"d"};
    {\ar@{->}@/^/ "a"; "b"};
    {\ar@{->}@/_/ "a"; "c"};
    {\ar@{->}@/^/ "b";"d"};
    {\ar@{->} "a"; "e"};
    {\ar@{->} "e";"b"};
    {\ar@{->} "e";"c"};
  \end{xy}
  $
\end{tabular}
\caption{A few zig-zag subdivisions.}\label{fig:zigzag}
\end{figure}

\begin{definition}[Diamond]
  Let $G$ be a directed graph. A pair $\alpha,\beta \in G(s,t)$ is
  called a \emph{diamond} if it does not admit a zig-zag subdivision.
\end{definition}

The idea behind the definition of a diamond is that any subdivision
containing a face that admits a zig-zag subdivision cannot be a
maximal subdivision. This is made precise in the following
proposition.

\begin{proposition}\label{prop:diamond}
  Let $G$ be an acyclic directed graph and $\alpha,\beta \in
  G(s,t)$. Every face of a maximal subdivision of
  $\langle\alpha,\beta\rangle$ is a diamond.  
\end{proposition}
\begin{proof}
  Let $G$ be an acyclic directed graph and let $(S,\varphi)$ be a maximal 
  subdivision of $\alpha,\beta \in G(s,t)$. By Theorem
  \ref{thm:power}, every face of $S$ has a unique source and
  target. That is, every face consists of a parallel pair of
  reductions $\eta,\psi:u \to v$. Suppose that $\langle
 \eta,\psi\rangle$ is a face of $S$ that is not a diamond. That
  is, it admits a zig-zag subdivision. So, we have
  \begin{align*}
    \eta &= u \stackrel{\eta_1}{\to} w \stackrel{\eta_2}{\to}
    v\\ 
    \psi &= u \stackrel{\psi_1}{\to} x \stackrel{\psi_2}{\to} v,
  \end{align*}
  and a zig-zag $\gamma$ between $w$ and $x$ that is a part of a
  subdivision of $\langle\eta, \psi\rangle$. By maximality, $\gamma$
  must be contained in $\S$. Since $\langle\eta, \psi\rangle$ is a
  face, $\varphi(|\gamma|)$ cannot lie in the region bounded by
  $\varphi(|\eta|)$ and $\varphi(|\psi|)$.  So, we are in one of the
  situations depicted in Figure \ref{fig:cantembed}.

  \begin{figure}[h]
    \begin{tabular}{ll}
      $
      \def\objectstyle{\scriptstyle}
      \def\labelstyle{\scriptstyle}
      \begin{xy}
        (0,0)*+{u}="u";
        (0,10)*+{w}="w";
        (0,-10)*+{x}="x";
        (15,0)*+{v}="v";
        {\ar@{-}@/^3pc/^{\gamma} "x";"w"};
        {\ar@{->}^{\eta_1} "u";"w"};
        {\ar@{->}_{\psi_1} "u"; "x"};
        {\ar@{->}@/_/_{\psi_2} "x"; "v"};
        {\ar@{->}@/^/^{\eta_2} "w";"v"};
      \end{xy}
      $

      &\qquad

      $
      \def\objectstyle{\scriptstyle}
      \def\labelstyle{\scriptstyle}
      \begin{xy}
        (0,0)*+{v}="v";
        (0,10)*+{w}="w";
        (0,-10)*+{x}="x";
        (-15,0)*+{u}="u";
        {\ar@{->}_{\psi_2} "x";"v"};
        {\ar@{->}^{\eta_2} "w";"v"};
        {\ar@{->}@/^/^{\eta_1} "u";"w"};
        {\ar@{->}@/_/_{\psi_1} "u";"x"};
        {\ar@{-}@/^3pc/^{\gamma} "w"; "x"};
      \end{xy}
      $
    \end{tabular}
    \caption{Possible embeddings of $\gamma$.}\label{fig:cantembed} 
  \end{figure}
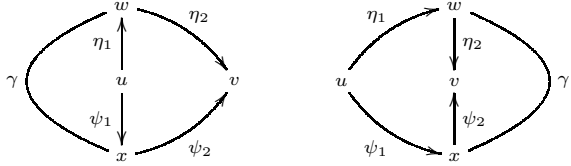
  
  Suppose that we are in the situation depicted in the left hand
  diagram of Figure \ref{fig:cantembed}. Since $\gamma$ is contained
  in $S$ and since $S$ is an st-graph, there is a path $s
  \stackrel{\delta}{\to} u$. By planarity, $\delta$ must factor
  through a vertex in $\gamma$ or $\eta_2$ or $\psi_2$. If $\delta$
  factors through a vertex in $\eta_2$ or $\psi_2$ then it is clear
  that $G$ contains a cycle, contradicting the fact that $G$ is
  acyclic. So, we must have $s \stackrel{\delta_1}{\to} z
  \stackrel{\delta_2}{\to} u$ for some vertex $z$ in
  $\gamma$. However, since $\gamma$ appears in a subdivision of
  $\langle \eta,\psi\rangle$, there is a path $u \stackrel{\zeta}{\to}
  z$ in $G$. Then, $\delta_2\cdot\zeta$ forms a cycle in $G$,
  contradicting the fact that $G$ is acyclic. So, $\gamma$ can not be
  embedded as in the left hand picture of Figure
  \ref{fig:cantembed}. Dually, it cannot be embedded as in the right
  hand picture of Figure \ref{fig:cantembed}.
  
  Therefore, $\langle\eta,\psi\rangle$ admits a zig-zag subdivision
  with zig-zag $\gamma$, contradicting the maximality of
  $(S,\varphi)$. So, $\langle \eta,\psi\rangle$ must be a diamond. 
\end{proof}

Combining Lemma \ref{lem:maximalface} and Proposition
\ref{prop:diamond}, we obtain our general Mac Lane Coherence theorem.

\begin{theorem}\label{thm:maclane}[Mac Lane Coherence]
  A finitely presented quasicycle-free structure $\S$ on a discrete
  category $\C$ is Mac Lane coherent if and only if every diamond in
  $\Red_{\S,\C}$ commutes in $\Fr_{\C}(\S)$. \qed
\end{theorem}

Theorem \ref{thm:maclane} says that in order to show that a
$2$-structure on a discrete category is Mac Lane coherent, we need to
do two things:
\begin{enumerate}
  \item Show that $\Fr_\C(\S)$ is quasicycle-free.
  \item Show that every diamond commutes.
\end{enumerate}

At the onset, showing that every diamond commutes can be a daunting
task. We can guide our investigations by  exploiting some
standard term rewriting theory \cite{terese:firstorder}. 

\begin{definition}
Let $\S$ be a $2$-structure on a category $\C$ and let $\varphi_1$ and
$\varphi_2$ be irreducible morphisms in $\Fr_\C(\S)$. We call
$\langle\varphi,\psi\rangle$ the \emph{initial span} in a diagram of the
following form:
\[
    \def\objectstyle{\scriptstyle}
    \def\labelstyle{\scriptstyle}
    \vcenter{
      \xymatrix{
        {\cdot} \ar[r]^{\varphi_1} \ar[d]_{\psi_1} &  {\cdot} \ar[d]^{\varphi_2} \\
        {\cdot} \ar[r]_{\psi_2} & {\cdot},
      }
    }
\]
\end{definition}

If $\varphi_1$ and $\psi_1$ are irreducible, then there are three
possibilities for a diamond with initial span $\langle \varphi,
\psi\rangle$:
\begin{enumerate}
  \item $\varphi$ and $\psi$ rewrite disjoint subterms. 
  \item $\varphi$ and $\psi$ rewrite nested subterms.
  \item $\varphi$ and $\psi$ rewrite overlapping
    subterms. Without loss of generality, we may assume that
    $\langle\varphi,\psi\rangle$ forms a critical peak. 
\end{enumerate}

By analogy with the critical pairs lemma \cite{terese:firstorder}, one
may hope to reduce the problem to only examining diamonds whose
initial span is a critical pair. Unfortunately, as the following two
examples show, there may be more than one diamond whose initial span
performs a given pair of nested or disjoint rewrites.

\begin{example}\label{ex:nested}
In this example we construct a terminating $2$-structure that has more
than one diamond with the same initial span performing a nested
rewrite. Let $\S$ be the $2$-structure consisting of unary functor 
symbols $I,J$ and $H$, together with the following reduction rules:
\begin{align*}
  I(x) &\to J(x)\\
  I(J(x)) &\to H(x)\\
  J(I(x)) &\to H(x)\\
\end{align*}
Let $\C$ be the discrete category generated by $\{A\}$ Then,
$\Fr_\C(\S)$ contains the following diagram:
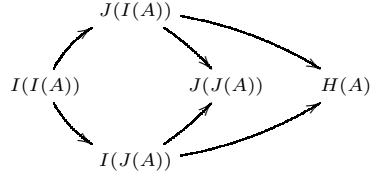
\[
\vcenter{
  \def\objectstyle{\scriptstyle}
  \def\labelstyle{\scriptstyle}
  \begin{xy}
    (0,0)*+{I(I(A))}="a";
    (12,10)*+{J(I(A))}="b";
    (12,-10)*+{I(J(A))}="c";
    (40,0)*+{H(A)}="d";
    (24,0)*+{J(J(A))}="e";
    {\ar@/^0.2pc/ "b";"e"};
    {\ar@/_0.2pc/ "c";"e"};
    {\ar@{->}@/_/ "c";"d"};
    {\ar@{->}@/^/ "a"; "b"};
    {\ar@{->}@/_/ "a"; "c"};
    {\ar@{->}@/^/ "b";"d"};
  \end{xy}
}
\]
 Since there is no reduction $J(J(A)) \to H(A)$, both parallel
 reductions form diamonds. 
\end{example}

\begin{example}\label{ex:disjoint}
  In this example we construct a terminating $2$-structure that has more
  than one diamond with the same initial span performing a disjoint
  rewrite. Let $\S$ be the $2$-structure consisting of unary functor 
  symbols $I$ and $J$, the binary functor symbol $\tensor$ and the
  following reduction rules: 
  \begin{align*}
    I(x) &\to J(x)\\
    J(x)\tensor I(x) &\to H(x)\\
    I(x)\tensor J(x) &\to H(x)\\
  \end{align*}
  Let $\C$ be the discrete category generated by $\{A\}$ Then,
  $\Fr_\C(\S)$ contains the following diagram:
  \[
  \vcenter{
    \def\objectstyle{\scriptstyle}
    \def\labelstyle{\scriptstyle}
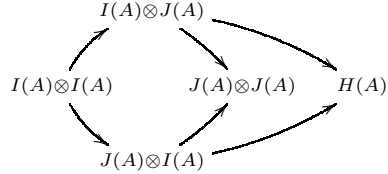
    \begin{xy}
      (0,0)*+{I(A)\tensor I(A)}="a";
      (12,10)*+{I(A) \tensor J(A)}="b";
      (12,-10)*+{J(A) \tensor I(A)}="c";
      (40,0)*+{H(A)}="d";
      (24,0)*+{J(A) \tensor J(A)}="e";
      {\ar@/^0.2pc/ "b";"e"};
      {\ar@/_0.2pc/ "c";"e"};
      {\ar@{->}@/_/ "c";"d"};
      {\ar@{->}@/^/ "a"; "b"};
      {\ar@{->}@/_/ "a"; "c"};
      {\ar@{->}@/^/ "b";"d"};
    \end{xy}
  }
  \]
   Since there is no reduction $J(A)\tensor J(A) \to H(A)$, both parallel
   reductions form diamonds. 
\end{example}

Examples \ref{ex:nested} and \ref{ex:disjoint} serve to warn us that
the collection of diamonds behaves a lot more subtly than the
collection of spans, which are the typical objects of study in
traditional term rewriting theory. Before illustrating the next subtle
point about quasicycle-free $2$-structures, we seperate those that are
inherently infinite from those that are inherently finite.

\begin{definition}[(Finitely) coherently axiomatisable]
  Let $\mathcal{R} := \langle\F,\theta_\F,\T\rangle$ be a term
  rewriting theory. We say that $\mathcal{R}$ is \emph{coherently
    axiomatisable} if there is a set of equations, $\theta_\T$,
  between reductions having the same source and target such that
  $\langle\tuple\rangle$ is a Mac Lane coherent $2$-structure. We say
  that $\mathcal{R}$ is \emph{finitely coherently axiomatisable} if it
  is finitely presented and there is a finite such $\theta_\T$.
\end{definition}

Theorem \ref{thm:maclane} immediately yields the following:

\begin{theorem}\label{thm:ca}
  A quasicycle-free $2$-structure is coherently axiomatisable.
\end{theorem}
\begin{proof}
  Add all diamonds as axioms and apply Theorem \ref{thm:maclane}.
\end{proof}

Since quasicycle-freeness was enough to guarantee only finitely many
subdivisions of a given parallel pair, one may hope that every
finitely presented quasicycle-free $2$-structure is finitely
coherently axiomatisable. Sadly, this is not the case.

\begin{proposition}\label{prop:nfca}
  There exist finitely presented coherently axiomatisable
   $2$-structures that are not finitely coherently axiomatisable. 
\end{proposition}
\begin{proof}
  Let $\S$ be the $2$-structure containing unary functor symbols
  $F,G,I$ and $H$, together with the following reduction rules:
  \begin{align*}
    I(x) &\to G(I(x))\\
    I(x) &\to F(I(x))\\
    F(x) &\to F(F(x))\\
    G(x) &\to G(G(x))\\
    F(x) &\to H(x)\\
    G(x) &\to H(x)
  \end{align*}
  Let $\C$ be the discrete category generated by $\{A\}$. It is clear
  that $\Fr_\C(\S)$ is quasicycle-free, so taking all diamonds as
  axioms, Theorem \ref{thm:maclane} implies that $\S$ is coherently
  axiomatisable. However, $\Fr_\C(\S)$ contains the following diagram: 
  \[
  \xymatrix@-0.2pc{
    & {G(I(A))} \ar[r]\ar[d] &{G^2(I(A))} \ar[r]\ar[d] & {G^3(I(A))}
    \ar[r]\ar[d]& {\dots}\\ 
    I(A) \ar@/^/[ur] \ar@/_/[dr] & {H(I(A))} & {H^2(I(A))} & {H^3(I(A))} &
    {\dots}\\ 
    & F(I(A)) \ar[r]\ar[u] & {F^2(I(A))}\ar[r]\ar[u] & {F^3(I(A))}
    \ar[r]\ar[u] & 
    {\dots} 
  }
  \]
  Since there are no reductions $H^i(A) \to H^j(A)$ for $i\ne j$, no
  finite collection of diamonds with source $I(A)$ implies the
  commutativity   of all others. So, $\S$ is not finitely coherently
  axiomatisable. 
\end{proof}

In this section, we have derived a very general Mac Lane coherence
theorem and used it to illuminate some of the many subtleties of
coherence problems for covariant structures. In the following section,
we apply this theory to a substantial coherence problem.

\section{Coherence for Iterated Monoidal Categories}\label{sect:imc}

Iterated monoidal categories were introduced in \cite{iterated}, in order
to make precise the intuition that the category of monoids internal to
a category corresponds to the space of loops internal to a topological
space. Iterating the construction of internal monoids, one arrives at
the concept of an $n$-fold monoidal category. The basic structure of
\cite{iterated} is to unpack the definition in terms of internal
monoids in order to obtain a categorical operad characterising $n$-fold
monoidal categories and to subsequently derive a weak homotopy
equivalence between the nerve of this operad and the little $n$-cubes
operad. 

The presentation of the operadic theory for iterated monoidal
categories in \cite{iterated} utilises strict associativity and unit
maps. Thus, there is a nontrivial congruence present at both the
object level and the structure level. This two-level structure leads
to a subtle interplay between the object-level equational theory and
the reductions. The coherence problem is further compounded by the
fact that $n$-fold monoidal categories do not have unique normal
forms. A coherence theorem is obtained in \cite{iterated},
which says that there is a unique map in an $n$-fold monoidal category
between two terms without repeated variables. The proof proceeds via
an intricate double induction on the number of variables and the
dimension of the outermost tensor product in the target of a
morphism. In this section, we exploit Theorem \ref{thm:maclane} to provide a
new, conceptual proof of the coherence theorem for iterated monoidal
categories.

\begin{definition}
  The $2$-structure for $n$-fold monoidal categories is denoted $\M_n$
  and consists of the following.
  \begin{enumerate}
    \item $n$ binary functor symbols:
      $$\tensor_1,\dots,\tensor_n: \C\times\C\to\C$$
    \item A nullary functor symbol $I$
    \item For $1 \le i \le n$:
          \begin{align*}
            A \tensor_i(B\tensor_iC) &= (A\tensor_i B)\tensor_i C\\
            A \tensor_i I = A\\
            I \tensor_I A = A\\   
          \end{align*}
    \item For each pair $(i,j)$ such that $1\le i < j \le n$, there is
      a reduction rule, called interchange:
      $$\eta^{ij}_{A,B,C,D}: (A\tensor_jB)\tensor_i(C\tensor_jD)\to
      (A\tensor_i C)\tensor_j (B\tensor_iD)$$
  \end{enumerate}
  The interchange rules are subject to the following
  conditions: 
  \begin{enumerate}
    \item Internal unit Condition: $\eta^{ij}_{A,B,I,I} =
      \eta^{ij}_{I,I,A,B} = id_{A\tensor_jB}$.
    \item External unit condition: $\eta^{ij}_{A,I,B,I} =
      \eta^{ij}_{I,A,I,B} = id_{A\tensor_i B}$.
    \item Internal associativity condition. The following diagram commutes:
      \[
      \def\objectstyle{\scriptstyle}
      \def\labelstyle{\scriptstyle}
      \vcenter{
        \xymatrix{
          {(A\tensor_jB)\tensor_i(C\tensor_jD)\tensor_i(E\tensor_jF)}
          \ar[rrr]^{\eta^{ij}_{A,B,C,D}\tensor_i id_{E\tensor_iF}}
          \ar[dd]^{id_{A\tensor_jB}\tensor_i\eta^{ij}_{C,D,E,F}}
          &&& 
          {((A\tensor_iC)\tensor_j(B\tensor_iD))\tensor_i(E\tensor_jF)}
          \ar[dd]^{\eta^{ij}_{A\tensor_iC,B\tensor_iD,E,F}}
          \\\\
          {(A\tensor_jB)\tensor_i((C\tensor_iE)\tensor_j(D\tensor_iF))} 
          \ar[rrr]_{\eta^{ij}_{A,B,C\tensor_iE,D\tensor_iF}}
          &&& 
          {(A\tensor_iC\tensor_iE)\tensor_j(B\tensor_iD\tensor_iF)}
        }
      }
      \]
    \item External associativity condition. The following diagram
      commutes:
      \[
      \def\objectstyle{\scriptstyle}
      \def\labelstyle{\scriptstyle}
      \vcenter{
        \xymatrix{
          {(A\tensor_jB\tensor_jC)\tensor_i(D\tensor_jE\tensor_jF)}
          \ar[rrr]^{\eta^{ij}_{A\tensor_jB,C,D\tensor_jE,F}}
          \ar[dd]^{\eta^{ij}_{A,B\tensor_jC,D,E\tensor_jF}}
          &&&
          {((A\tensor_jB)\tensor_i(D\tensor_jC))\tensor_j(C\tensor_iF)}
          \ar[dd]^{\eta^{ij}_{A,B,D,C}\tensor_jid_{C\tensor_iF}}
          \\\\
          {(A\tensor_iD)\tensor_j((B\tensor_jC)\tensor_i(E\tensor_jF))}
          \ar[rrr]^{id_{A\tensor_iD}\tensor_j\eta^{ij}_{B,C,E,F}}
          &&&
          {(A\tensor_iD)\tensor_j(B\tensor_iE)\tensor_j(C\tensor_iF)}
        }
      }
      \]
    \item Giant hexagon condition. The following diagram commutes:
      \[
      \def\objectstyle{\scriptstyle}
      \def\labelstyle{\scriptstyle}
      \vcenter{
        \begin{xy}
           (35,65)*+{((A\tensor_kB)\tensor_j(C\tensor_kD))\tensor_i
             ((E\tensor_kF)\tensor_j(G\tensor_kH))}="a";
          (0,45)*+{((A\tensor_jC)\tensor_k(B\tensor_jD))\tensor_i
            ((E\tensor_jG)\tensor_k(F\tensor_jH))}="b";
          (70,45)*+{((A\tensor_kB)\tensor_i(E\tensor_kF))\tensor_j
            ((C\tensor_kD)\tensor_i(G\tensor_kH))}="c";
          (0,20)*+{((A\tensor_jC)\tensor_i(E\tensor_jG))\tensor_k
            ((B\tensor_jD)\tensor_i(F\tensor_jH))}="d";
          (70,20)*+{((A\tensor_iE)\tensor_k(B\tensor_iF))\tensor_j
            ((C\tensor_iG)\tensor_k(D\tensor_iH))}="e";
          (35,0)*+{((A\tensor_iE)\tensor_j(C\tensor_iG))\tensor_k
            ((B\tensor_iF)\tensor_j(D\tensor_iK)}="f";
          {\ar@{->}^{\eta^{ik}\tensor_j\eta^{ik}} "c";"e"}
          {\ar@{->}@/_/_{\eta^{ij}\tensor_k\eta^{ij}} "d";"f"}
          {\ar@{->}@/^/^{\eta^{ij}} "a";"c"}
          {\ar@{->}@/_/_{\eta^{jk}\tensor_k\eta^{jk}} "a";"b"}
          {\ar@{->}_{\eta^{ik}} "b";"d"}
          {\ar@{->}@/^/^{\eta^{jk}} "e";"f"}
        \end{xy}
      }
      \]
      In the giant hexagon, $(i,j,k)$ is such that $1 \le i < j < k
      \le n$ and the natural transformations have the evident
      components. 
      \end{enumerate}
\end{definition} 

It is very easy to characterise those reductions in $\M_n$ that are
general position.

\begin{lemma}
  Let $\C$ be a discrete category. A reduction $s \to t$ in
  $\Fr_\C(\M_n)$ is in general position if and only if $s$ and
  $t$ contain no repeated variables. \qed
\end{lemma}

Because of the fact that an $n$-fold monoidal category is strictly
associative and has a strict unit, we can derive various useful maps
via Eckmann-Hilton style arguments. Several of these maps will be of
particular use to us. In the following, we assume that $(i,j)$ is such
that $1 \le i < j \le n$. The derived maps are as follows:

\begin{enumerate}
 \item Dimension raising: $
   \xymatrix@1{A\tensor_iB \ar[r]^{\iota^{ij}_{A,B}} & A \tensor_j
     B.}$ This represents the following composition:
   \[
   \def\objectstyle{\scriptstyle}
   \def\labelstyle{\scriptstyle}
   \vcenter{
     \xymatrix@1{
       {A\tensor_iB}\ar[r]^-{=} & {(A\tensor_j
         I)\tensor_i(I\tensor_jB)}
       \ar[r]^-{\eta^{ij}} & {(A\tensor_i I)\tensor_j
         (I\tensor_i B)} \ar[r]^-{=} & {A\tensor_j B}
     }
   }
   \]
 \item Twisted dimension raising: $
   \xymatrix@1{A\tensor_iB \ar[r]^{\tau^{ij}_{A,B}} & B \tensor_j
     A.}$ This represents the following composition:
   \[
   \def\objectstyle{\scriptstyle}
   \def\labelstyle{\scriptstyle}
   \vcenter{
     \xymatrix@1{
       {A\tensor_iB}\ar[r]^-{=} & {(I\tensor_j
         A)\tensor_i(I\tensor_jB)}
       \ar[r]^-{\eta^{ij}} & {(I\tensor_i B)\tensor_j
         (A\tensor_i I)} \ar[r]^-{=} & {B\tensor_j A}
     }
   }
   \]
   \item Left weak distributivity (This name is chosen to reflect the
     connection with weakly distributive categories \cite{CockettSeely:wdc}): $
   \xymatrix@1{{A\tensor_i(B\tensor_jC)} \ar[r]^{\delta^{ij}_{A,B,C}} &
     {(A\tensor_iB)\tensor_jC}.}$ This represents the following
   composition: 
      \[
   \def\objectstyle{\scriptstyle}
   \def\labelstyle{\scriptstyle}
   \vcenter{
     \xymatrix@1{
       {A\tensor_i(B\tensor_j C)}\ar[r]^-{=} & {(A\tensor_j
         I)\tensor_i(B\tensor_jC)}
       \ar[r]^-{\eta^{ij}} & {(A\tensor_i B)\tensor_j
         (I\tensor_i C)} \ar[r]^-{=} & {(A\tensor_iB)\tensor_j C}
     }
   }
   \]
   \item Twisted left weak distributivity:$
   \xymatrix@1{{A\tensor_i(B\tensor_jC)} \ar[r]^{\tilde{\delta}^{ij}_{A,B,C}} &
     {B\tensor_j(A\tensor_iC)}.}$ This represents the following
   composition:
   \[
   \def\objectstyle{\scriptstyle}
   \def\labelstyle{\scriptstyle}
   \vcenter{
     \xymatrix@1{
       {A\tensor_i(B\tensor_j C)}\ar[r]^-{=} & {(I\tensor_j
         A)\tensor_i(B\tensor_jC)}
       \ar[r]^-{\eta^{ij}} & {(I\tensor_i B)\tensor_j
         (A\tensor_i C)} \ar[r]^-{=} & {B\tensor_j(A\tensor_iC)}
     }
   }
   \]
 \item Right weak distributivity: $
   \xymatrix@1{{(A\tensor_jB)\tensor_iC} \ar[r]^{\gamma^{ij}_{A,B,C}} &
     {A\tensor_j(B\tensor_iC)}.}$This represents the following
   composition:
   \[
   \def\objectstyle{\scriptstyle}
   \def\labelstyle{\scriptstyle}
   \vcenter{
     \xymatrix@1{
       {(A\tensor_jB)\tensor_i C}\ar[r]^-{=} & {(A\tensor_j
         B)\tensor_i(I\tensor_jC)}
       \ar[r]^-{\eta^{ij}} & {(A\tensor_i I)\tensor_j
         (B\tensor_i C)} \ar[r]^-{=} & {A\tensor_j(B\tensor_i C)}
     }
   }
   \]
    \item Twisted right weak distributivity: $
   \xymatrix@1{{(A\tensor_iB)\tensor_jC}
     \ar[r]^{\tilde{\gamma}^{ij}_{A,B,C}} & {(A\tensor_jC)\tensor_iB}.}$
   This represents the following composition:
   \[
   \def\objectstyle{\scriptstyle}
   \def\labelstyle{\scriptstyle}
   \vcenter{
     \xymatrix@1{
       {(A\tensor_iB)\tensor_j C}\ar[r]^-{=} & {(A\tensor_i
         B)\tensor_j(C\tensor_iI)}
       \ar[r]^-{\eta^{ij}} & {(A\tensor_j C)\tensor_i
         (B\tensor_j I)} \ar[r]^-{=} & {(A\tensor_jC)\tensor_i B}
     }
   }
   \]
\end{enumerate}

With the above maps, it is easy to see that iterated monoidal
categories do not have unique normal forms.

\begin{lemma}
  Let $\C$ be a discrete category. Then, $\Fr_\C(\M_n)$ is not confluent.
\end{lemma}
\begin{proof}
  The following span is clearly not joinable:
  \[
  \def\objectstyle{\scriptstyle}
  \def\labelstyle{\scriptstyle}
  \xymatrix{
    {A\tensor_i B} \ar[r]^{\iota^{in}_{A,B}} \ar[d]_{\tau^{in}_{A,B}}
    & {A \tensor_n B}\\
    {B \tensor_n A}
  }
  \]
\end{proof}

Our first step is to bring iterated monoidal categories into the realm
of applicability of Theorem \ref{thm:maclane}.

\begin{proposition}\label{prop:imcterm}
  $\M_n$ is terminating.
\end{proposition}
\begin{proof}
  Let $\C$ be a discrete category. We shall construct a ranking
  function on $\Fr_\C(\M_n)$. Define $\hat{\rho}:\Term_{\Ob{\C}}(\M_n)
  \to \N$ by:
  \[
  \hat{\rho}(t) = 
  \begin{cases} 
    i + \hat{\rho}(A) + 2\hat{\rho}(B) & \text{if $t = A\tensor_iB$}\\
    0 & \text{if $t = I$}
  \end{cases}
  \]
  At the moment, $\hat{\rho}$ is not a ranking function, since it is
  sensitive to the order of parenthesisation and the presence of
  units. We can, however, use it to construct a ranking function. Let
  $[t]$ be an object in $\Fr_\C(\M_n)$. Define:
  \[
  \rho([t]) = \min \{\hat{\rho}(t') : t' \in [t]\}
  \]
  The map $\rho$ effectively calculates the rank of the member of a
  congruence class $[t]$ which has no units and the left most
  bracketing. It is immediate from the definition that, for $t_1,t_2
  \in [t]$, we have $\rho(t_1) = \rho(t_2)$. We now need to show that
  interchange reduces the rank and must be careful to check the maps
  arising from Eckmann-Hilton arguments also. Let $j > i$. 
  \begin{itemize}
    \item Interchange:
    \begin{align*}
      \rho((A\tensor_j B) \tensor_i (C \tensor_j D)) &=
      i + 2j + \rho(A) + \rho(B) + \rho(C) + \rho(D)\\
      \rho((A\tensor_i C)\tensor_j (B \tensor_iC)) &=
      2i + j + \rho(A) + \rho(B) + \rho(C) + \rho(D)
    \end{align*}
    Since $j> i$, we have $i + 2j > 2i + j$.
  \item Left linear distributivity:
    \begin{align*}
      \rho(A\tensor_i(B\tensor_j C)) &= 
      i + 2j + \rho(A) + 2\rho(B) + 4\rho(C)\\
      \rho((A\tensor_i B) \tensor_j C) &= 
      i + j + \rho(A) + 2\rho(B) + 2\rho(C)
    \end{align*}
  \end{itemize}
  The other cases are handled similarly. It follows that $\rho$ is a
  ranking function on $\Fr_\C(\M_n)$, so $\Fr_\C(\M_n)$ is terminating.
\end{proof}

It follows from Proposition \ref{prop:imcterm} and Theorem
\ref{thm:ca} that $\M_n$ is coherently axiomatisable. At this stage,
however, we don't even know whether it is finitely coherently
axiomatisable. Before examining diamonds in $\M_n$ for commutativity,
we recall some useful terminology and results from \cite{iterated}. 

Let $A$ be an object in $\Fr_\C(\M_n)$. For a set $X \subseteq \var(A)$, we
write $A - X$ to denote the object resulting from substituting
$I$ for each variable in $X$. For instance $(A\tensor_i
B)\tensor_j(C \tensor_i E) - \{B,D\} = A\tensor_j C$. We say that a
term $B$ is \emph{in} a term $A$ and write $B \in A$ if there is some
$X \subseteq \var(A)$ such that $A - X = B$. Of crucial importance to
us is the following result of \cite{iterated}.

\begin{theorem}[\cite{iterated}]\label{thm:maps}
  Let $\C$ be a discrete category and let $A$ and $B$ be objects of
  $\Fr_\C(\M_n)$. A necessary and sufficient condition for the
  existence of a map $A \to B$ in $\Fr_\C(\M)$ is that, for each $a,b
  \in \var(A)$, if $a\tensor_i b \in A$, then one of the following holds:
  \begin{itemize}
    \item There is some $j \ge i$ such that $a \tensor_j b \in B$
    \item There is some $j > i$ such that $b \tensor_j a \in B$
  \end{itemize}
\end{theorem}

Theorem \ref{thm:maps} gives us the technical tool that we need in
order to show that various parallel pairs of maps are not diamonds. We
begin our analysis of the collection of diamonds of $\Fr_\C(\M_n)$
with diamonds whose initial span rewrites disjoint pieces of a term.

\begin{lemma}\label{lem:itdisjoint}
  Let $A \tensor_i B \in \Fr_\C(\M_n)$ and suppose that there are maps
  $\varphi : A \to A'$ and $\psi:B \to B'$. Then, in the following
  diagram, the square labelled (d) is a commutative diamond and there
  is a map $A'\tensor_i B' \to C$:
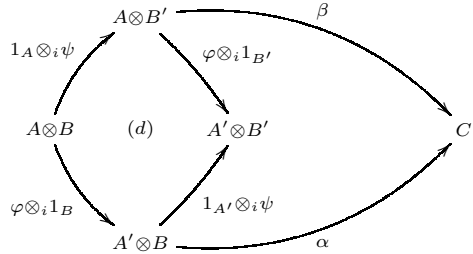
  \[
  \vcenter{
    \def\objectstyle{\scriptstyle}
    \def\labelstyle{\scriptstyle}
    \begin{xy}
      (0,0)*+{A\tensor B}="a";
      (12,15)*+{A \tensor B'}="b";
      (12,-15)*+{A' \tensor B}="c";
      (55,0)*+{C}="d";
      (25,0)*+{A'\tensor B'}="e";
      (12,0)*+{(d)};
      {\ar@/^0.2pc/^{\varphi \tensor_i 1_{B'}} "b";"e"};
      {\ar@/_0.2pc/_{1_{A'} \tensor_i \psi} "c";"e"};
      {\ar@{->}@/_1.5pc/_{\alpha} "c";"d"};
      {\ar@{->}@/^/^{1_A\tensor_i \psi} "a"; "b"};
      {\ar@{->}@/_/_{\varphi \tensor_i 1_B} "a"; "c"};
      {\ar@{->}@/^1.5pc/^{\beta} "b";"d"};
    \end{xy}
  }
  \]
\end{lemma}
\begin{proof}
  The square labelled (d) commutes by functoriality and it is easy to
  see that it does not admit a zig-zag subdivision, so it is a
  diamond. The tricky part is showing the existence of a map
  $A'\tensor_i B' \to C$. 

  Let $X,Y \in \var(A'\tensor_i B')$ and suppose that $X \tensor_k Y
  \in A'\tensor_i B'$. There are a few cases to 
  consider.
  \begin{itemize}
    \item If $X,Y \in A'$, then $\alpha$ implies that there is some $m
      \ge k$ such that $X \tensor_m Y \in C$ or there is some $m> k$ such
      that $Y \tensor_m X \in C$.
    \item If $X,Y \in B'$, then $\beta$ implies that there is some $m
      \ge k$ such that $X \tensor_m Y\in C$ or there is some $m> k$ such
      that $Y \tensor_m X \in C$. 
    \item If $X\in A'$ and $Y \in B'$, then $X \tensor_i Y \in A'
      \tensor B$. So, by $\alpha$, there is some $m \ge i$ such that $X
      \tensor_m Y \in C$ or there is some $m> i$ such that $Y \tensor_m X
      \in C$  
  \end{itemize}
Putting all of the above facts together, it follows from Theorem
\ref{thm:maps} that there is a map $A'\tensor_i B' \to C$.
\end{proof}

Our next port of call is diamonds whose initial span rewrites nested
subterms. For a term $A$ and a subterm $B \le A$, we write $A[B]$ to
represent this nested term. 

\begin{lemma}\label{itnested}
  Let $A[B] \in \Fr_\C(\M_n)$ and suppose that there are maps
  $\varphi : A[B] \to A'[B]$ and $\psi:B \to B'$. Then, in the following
  diagram, the square labelled (d) is a commutative diamond and there
  is a map $A'[B'] \to C$: 
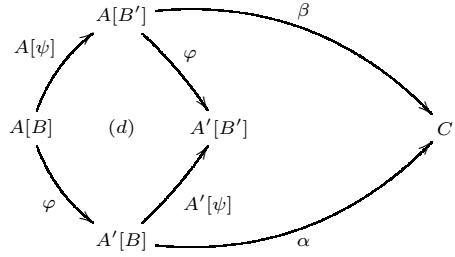
  \[
  \vcenter{
    \def\objectstyle{\scriptstyle}
    \def\labelstyle{\scriptstyle}
    \begin{xy}
      (0,0)*+{A[B]}="a";
      (12,15)*+{A[B']}="b";
      (12,-15)*+{A'[B]}="c";
      (55,0)*+{C}="d";
      (25,0)*+{A'[B']}="e";
      (12,0)*+{(d)};
      {\ar@/^0.2pc/^{\varphi} "b";"e"};
      {\ar@/_0.2pc/_{A'[\psi]} "c";"e"};
      {\ar@{->}@/_1.5pc/_{\alpha} "c";"d"};
      {\ar@{->}@/^/^{A[\psi]} "a"; "b"};
      {\ar@{->}@/_/_{\varphi} "a"; "c"};
      {\ar@{->}@/^1.5pc/^{\beta} "b";"d"};
    \end{xy}
  }
  \]
\end{lemma}
\begin{proof}
  The square labelled (d) commutes by naturality. The rest of the
  proof is similar to that of Lemma \ref{lem:itdisjoint}.
\end{proof}

We now know that the only nontrivial diamonds in $\Fr_C(\M_n)$ have a
critical pair as their initial span. Our remaining task is to perform
a critical pairs analysis on $\M_n$.

\subsection{Interchange + associativity}

Let $j > i$. The first way in which interchange and associativity can
interact is in the term $X \tensor_i (C\tensor_j D) \tensor_i
(E\tensor_j F)$. Without loss of generality, we may assume that $X = A
\tensor_j B$, because we could always take $X = X \tensor_j I$. The
resulting span then gets completed into the internal associativity
axiom. One may then apply Theorem \ref{thm:maps} to show that there is
no other diamond with the same initial span.

The second way in which interchange can interact with associativity is
in the term $(A\tensor_j B)\tensor_i (C\tensor_j D\tensor_j E)$. In
this case, we get the following square, where the labels have the
evident components.

\[
\def\objectstyle{\scriptstyle}
\def\labelstyle{\scriptstyle}
\xymatrix{
  {(A\tensor_j B)\tensor_i (C\tensor_j D \tensor_j E)} \ar[rr]^{\eta}
  \ar[d]_{\eta} && {(A\tensor_i (C\tensor_j D))\tensor_j (B\tensor_i
    E)}\ar[d]^{\delta\tensor_j 1}\\
    {(A\tensor_i C)\tensor_j (B\tensor_i (D\tensor_j E))}
    \ar[rr]_{1\tensor_j \tilde{\delta}} &&
    {(A\tensor_i C)\tensor_j D \tensor_j (B\tensor_i E)}
}
\]

The above square commutes by substituting $(A\tensor_j I \tensor_j
B)\tensor_i (C\tensor_j D\tensor_j E)$ for the source and using
the external associativity axiom. Theorem \ref{thm:maps} easily yields
that there can be no other diamonds with the same initial span.

Similarly, a critical pair arises at $(A\tensor_j B\tensor_j
C)\tensor_i(D\tensor_j E)$. The analysis is analogous to the previous
case by inserting a unit to obtain $(A\tensor_j B \tensor_j
C)\tensor_i (D\tensor_j I\tensor_j E)$.

\subsection{Interchange + interchange}
Let $i < j < k$. An overlap between interchange rules occurs at $(A\tensor_j
B)\tensor_i ((C\tensor_k D)\tensor_j (E\tensor_j F))$. Since we have
strict units, we may assume that $A = A_1 \tensor_k A_t$ and $B = B_1
\tensor_k B_2$. We then obtain the initial span of the giant hexagon
axiom. The hexagon forms a diamond and it follows from Theorem
\ref{thm:maps} that there are no other diamonds with this initial
span.

\subsection{Interchange + units}

The critical pairs arising from the interaction of interchange with
units yield the various Eckmann-Hilton maps. As we have seen, these
are not always joinable. When they are, they commute by the following
lemma. 

\begin{lemma}
  Let $\C$ be a discrete category. The following diagrams commute in
  $\Fr_\C(\M_n)$, where $1 \le i < j < k \le n$:
  \begin{center}
    \begin{tabular}{cc}
      $
      \def\objectstyle{\scriptstyle}
      \def\labelstyle{\scriptstyle}
      \xymatrix{
        &{A\tensor_j B} \ar[dr] & {} \\
        {A\tensor_i B} \ar[ur] \ar[rr] \ar@{}[urr]|{(1)} && {A\tensor_k B}
      }$

      &

      $
      \def\objectstyle{\scriptstyle}
      \def\labelstyle{\scriptstyle}
      \xymatrix{
        &{B\tensor_j A} \ar[dr] & {} \\
        {A\tensor_i B} \ar[ur] \ar[rr] \ar@{}[urr]|{(2)} && {A\tensor_k B}
      }$\\

      $
      \def\objectstyle{\scriptstyle}
      \def\labelstyle{\scriptstyle}
      \xymatrix{
        &{A\tensor_j B} \ar[dr] & {} \\
        {A\tensor_i B} \ar[ur] \ar[rr] \ar@{}[urr]|{(3)} && {B\tensor_k A}
      }$

      &

      $
      \def\objectstyle{\scriptstyle}
      \def\labelstyle{\scriptstyle}
      \xymatrix{
        &{B\tensor_j A} \ar[dr] & {} \\
        {A\tensor_i B} \ar[ur] \ar[rr] \ar@{}[urr]|{(4)} && {B\tensor_k A}
      }$
    \end{tabular}
  \end{center}
\end{lemma}
\begin{proof}
  This follows from \cite[Lemma 4.22]{iterated}. More explicitly it
  follows from the giant hexagon axiom by making the following
  substitutions: 

  \bigbreak
  
  \begin{enumerate}
    \item $A\tensor_iB = ((A\tensor_k I)\tensor_j(I\tensor_k I))
      \tensor_i ((I\tensor_kI)\tensor_j(I\tensor_k B))$
      \medbreak
    \item $A\tensor_iB = ((I\tensor_k I)\tensor_j(A\tensor_k I))
      \tensor_i ((I\tensor_k B)\tensor_j(I\tensor_k I))$
      \medbreak
    \item $A\tensor_iB = ((I\tensor_k A)\tensor_j (I\tensor_k I))
      \tensor_i ((I\tensor_k I) \tensor_j (B\tensor_k I))$
      \medbreak
    \item $A\tensor_iB = ((I\tensor_k I)\tensor_j (I\tensor_k A))
      \tensor_i ((B\tensor_k I) \tensor_j (I\tensor_k I))$
  \end{enumerate}
\end{proof}

\subsection{Putting it all together}

We have seen that $\Fr_\C(\M_n)$ is terminating and that every diamond
in $\Fr_\C(\M_n)$ commutes. We can therefore apply Theorem
\ref{thm:maclane} to obtain the coherence theorem for iterated
monoidal categories.

\begin{theorem}
  Let $\C$ be a discrete category. If $A$ and $B$ are objects of
  $\Fr_\C(\M_n)$ having no repeated variables, then there is at most one
  map $A \to B$ in $\Fr_\C(\M_n)$. \qed
\end{theorem}

\section{Conclusions}

Both of our general coherence theorems, Theorem \ref{thm:lambek} and
Theorem \ref{thm:maclane}, rely on 
the underlying structure being quasicycle-free. One may well call this
condition into question and wonder whether we can get away with a
weaker condition. For Lambek Coherence, quasicycle-freeness does not
capture all covariant structures known to be Lambek coherent. For
example, braided monoidal categories are certainly not quasicycle free
and yet they are well known to be Lambek coherent. However, the method
for proving this adds a rewrite system to the reductions, thus
expanding the amount of information available. The question still
stands, then, of whether there is a property \emph{of the underlying
  term rewriting system} that leads to Lambek coherence for
non-quasicycle-free structures.

The reliance on quasicycle-freeness for Mac Lane coherence seems more
fundamental. However, two of the crucial ingredients of our theory,
Theorem \ref{thm:power} and Proposition \ref{prop:diamond} rely solely
on acyclicity. This leads us to ask whether we can find conditions on
an acyclic $2$-structure that ensure Mac Lane coherence. 

Nevertheless, our focus on quasicycle-free structures has proven to be
broad and powerful enough for us to find the conceptual reason for the
coherence theorem for iterated monoidal categories. Moreover, it has
allowed us to show that there is a wide variety of coherence
phenomena, even in the purely covariant case. 

\section{Acknowledgements}

Thanks to the members of The Australian Category Seminar, in
particular Mike Johnson and Michael Batanin, for useful discussion and
feedback on aspects of this work.

\end{document}